\documentclass{article}

\usepackage{amsmath,amsthm,amssymb,latexsym,libertinus}

\newtheorem{theorem}{Theorem}[section]
\newtheorem{proposition}[theorem]{Proposition}
\newtheorem{lemma}[theorem]{Lemma}
\newtheorem{corollary}[theorem]{Corollary}
\newtheorem{remark}[theorem]{Remark}

\newtheorem{example}[theorem]{Example}
\newtheorem{assumption}[theorem]{Assumption}

\title{Rate estimates for total variation distance with applications\thanks{The 
author gratefully acknowledges the support of 
the National Research, Development and Innovation Office (NKFIH) through grant K 143529
and also within the framework of the Thematic Excellence Program 2021 (National Research subprogramme 
``Artificial intelligence, large networks, data security: mathematical foundation and applications'').
The author would like to thank the Isaac Newton Institute for Mathematical
Sciences (EPSRC Grant Number: EP/V521929/1) for the support and hospitality during the programme ``Diffusions in
machine learning: Foundations, generative models and non-convex optimisation''
where parts of this work have been presented and discussed.}}

\author{
Mikl\'os R\'asonyi\thanks{HUN-REN Alfr\'ed R\'enyi Institute of Mathematics and E\"otv\"os Lor\'and University, Budapest, 
Hungary.}} 

\date{\today}

\begin{document}

\maketitle{}

\begin{abstract}
{We present a Fourier-analytic method for estimating convergence rates in total variation distance in terms
of various metrics related to weak convergence. Applications are provided in
the areas of Malliavin calculus,
normal approximation and stochastic dynamical systems with memory.}
\end{abstract}



\section{Introduction}

Convergence rates in probability theory are often given by metrics related to weak convergence (Prokhorov,
Wasserstein, Fortet--Mourier, etc.). In this paper we propose a method to 
transfer such results to total variation distance estimates, under suitable moment- and
smoothness conditions on the random variables involved. 

Our main results are
presented in Section \ref{tot}: Lemma \ref{withrate} establishes a relationship between
convergence rates in Fortet-Mourier metric and those in total variation distance. These rates
depend, on one hand, on the existence of higher moments, on the other hand, on the 
smoothness of these random variables as measured by the tail decay of their characteristic functions.
If all moments exist and characteristic functions decay faster than any (negative) power function then
almost the same rate can be shown for total variation as for the Fortet-Mourier metric: this is the 
content of Corollary \ref{fourier2}.

In Section \ref{contra}, we show how this method applies to the ergodic theory of 
a subclass of so-called stochastically recursive sequences (see Chapter 3 of \cite{borovkov}).
We study stochastic difference equations driven by non-i.i.d. (coloured) noise.
It is well-known that such systems converge weakly to a limiting law under contractivity assumptions.
Using our tools from Section \ref{tot} we deduce total variation convergence at a geometric speed for
such systems in Theorem \ref{lime}.

In Section \ref{berry}, we revisit the multidimensional Berry--Esseen theorem and prove almost
optimal convergence rate in total variation in Theorem \ref{be}, under reasonable conditions.
Our result covers several cases which previous papers do not.

Section \ref{malliavin}
presents an application to sequences of smooth functionals in Malliavin calculus. 
We prove a slight strengthening of a result of \cite{bc2} in Theorem \ref{smoothsequence1} using a 
completely different argument, based on Corollary
\ref{fourier2}.

Finally, Lemma \ref{exx} in Section \ref{ramif} shows that, under even more stringent moment and
smoothness conditions one may find even stronger rate estimates. 

All proofs will be given in Section \ref{pofs}.
Our results open several further research directions. In particular, we are currently working on
providing convergence guarantees for certain machine learning algorithms (such as the stochastic gradient Langevin
dynamics, see \cite{6} and the references therein).
In this area, results are customarily given in Wasserstein distance, our purpose is to obtain
total variation estimates as well.

\section{Conditions for convergence in total variation}\label{tot}


We fix an integer $d\geq 1$ and work in the Euclidean space $\mathbb{R}^{d}$ with scalar product $\langle \cdot,\cdot\rangle$,{}
equipped with its Borel sigma-algebra $\mathcal{B}(\mathbb{R}^{d}$. For an $\mathbb{R}^{d}$-valued random variable $X$, its law will be denoted by $\mathcal{L}(X)$
and its characteristic function by
$$
\phi_{X}(u):=E[e^{i\langle u,X\rangle}],\ u\in\mathbb{R}^{d}.
$$

For each integer $m\geq 1$, the Euclidean norm of a vector $x\in\mathbb{R}^{m}$ will be denoted by $|x|$, where $m$ will always be clear from the context.
For a function $f:\mathbb{R}^{d}\to\mathbb{R}^{m}$ we denote $||f||_{\infty}:=\sup_{x\in\mathbb{R}^{d}}|f(x)|$. Again, $m$ will always be
given by the context.

We denote by $\mathcal{M}_{1}$ the set of continuously differentiable functions $f:\mathbb{R}^{d}\to\mathbb{R}$
such that $||f||_{\infty}+||\nabla f||_{\infty}\leq 1$. Note that $\nabla f:\mathbb{R}^{d}\to\mathbb{R}^{d\times d}$ in this
case. Following \cite{bally,bc1}, for arbitrary $\mathbb{R}^{d}$-valued Borel probability measures $\mu_{1},\mu_{2}$ 
we define their \emph{Fortet-Mourier distance} by 
$$
d_{FM}(\mu_{1},\mu_{2}):=\sup_{f\in\mathcal{M}_{1}}\left|\int_{\mathbb{R}^{d}}f(x)\mu_{1}(dx)-\int_{\mathbb{R}^{d}}f(x)\mu_{2}(dx)\right|.
$$
For $\mathbb{R}^{d}$-valued random variables $X_{1},X_{2}$ we will write $d_{FM}(X_{1},X_{2})$ when we indeed mean $d_{FM}(\mathcal{L}(X_{1}),
\mathcal{L}(X_{2}))$. 
The \emph{total variation distance} of $\mu_{1},\mu_{2}$ is
$$
d_{TV}(\mu_{1},\mu_{2}):=\sup_{|f|\leq 1}\left|\int_{\mathbb{R}^{d}}f(x)\mu_{1}(dx)-\int_{\mathbb{R}^{d}}f(x)\mu_{2}(dx)\right|,
$$
where the supremum ranges over measurable functions with absolute value at most one. Again, $d_{TV}(X_{1},X_{2})$ means
$d_{TV}(\mathcal{L}(X_{1}),
\mathcal{L}(X_{2}))$.

\begin{remark}{\rm The metric $d_{FM}$ is more often defined using Lipschitz functions, with $\mathcal{M}_{1}$ replaced by 
$$
\mathcal{N}_{1}:=\left\{f:\, ||f||_{\infty}\leq 1,\ \sup_{x\neq y}\frac{|f(x)-f(y)|}{|x-y|}\leq 1\right\},
$$
see e.g.\ page 210 of \cite{nourdin-peccati}. The two definitions are easily seen to be equivalent.
}
\end{remark}

Using characteristic functions for estimating total variation has been much used in the 
literature, see e.g.\ \cite{bobkov,bcg}. The next lemma is a novel contribution to this array of
techniques.

\begin{lemma}\label{withrate}
Let $X_{1}$, $X_{2}$ be random variables, let $\phi_{1},\phi_{2}$ denote the corresponding
characteristic functions. Assume that
\begin{equation}\label{fi}
|\phi_{i}(u)|\leq c_{\phi}(1+|u|)^{-d-\gamma},\quad i=1,2,	
\end{equation}
hold for some $c_{\phi},\gamma>0$ and also
\begin{equation}\label{ef}
E[|X_{i}|^{\delta}]\leq c_{f},\quad i=1,2,	
\end{equation}
for some $c_{f},\delta>0$. Then
\begin{equation}\label{alenyeg}
d_{TV}(X_{1},X_{2})\leq G d_{FM}^{g}(X_{1},X_{2})	
\end{equation}
holds for 
$$
g=\frac{\gamma\delta}{(d+1+\gamma)(d+\delta)}
$$
and for some constant $G>0$ which depends only on $d$, $\gamma$, $\delta$, $c_{\phi}$ and $c_{f}$. 
The random variables have respective density functions $f_{1},f_{2}$ and
\begin{equation}\label{sup}
\sup_{x\in\mathbb{R}^{d}}|f_{1}(x)-f_{2}(x)|\leq \bar{G} d_{FM}^{\bar{g}}(X_{1},X_{2})
\end{equation}
holds for $\bar{g}=\frac{\gamma}{(d+1+\gamma)}$ and for a suitable $\bar{G}$, depending only
on $d$, $\gamma$ and $c_{\phi}$.
\end{lemma}

If all the moments are bounded and the laws are smooth enough (in the sense that the characteristic function
decays to $0$ faster than any negative power function) then convergence rates in total variation are
arbitrarily close to those in $d_{FM}$, as the next Corollary shows. 

\begin{corollary}\label{fourier2}
Let $E[|X_{1}|^{k}]+E[|X_{2}|^{k}]\leq K_{k}<\infty$ for all $k\in\mathbb{N}$. Assume also that for all $l\in\mathbb{N}$ there is
$C_{l}>0$ such that the corresponding characteristic functions $\phi_{1},\phi_{2}$ satisfy
\begin{equation}\label{pipi}
|\phi_{1}(u)|+|\phi_{2}(u)|\leq \frac{C_{l}}{(1+|u|)^{l}},\ u\in\mathbb{R}^{d}.
\end{equation}
Then for all $0<\epsilon<1$ there is a constant $C_{\epsilon}$
(which depends on $d$ and on the sequences $K_{k},C_{l}$ but not on $X_{1},X_{2}$)
such that
$$
d_{TV}(X_{1},X_{2})\leq C_{\epsilon}d_{FM}^{1-\epsilon}(X_{1},X_{2}).
$$
\end{corollary}

\begin{remark} In the special case $d=1$ a closely related result
has already been shown in \cite{che} with a different approach, not relying on Fourier analysis. 
Theorem 2.1 of \cite{che} directly implies that, if $X_{1}$, $X_{2}$ have $C^{\infty}$ densities 
$f_{1},f_{2}$ with all their derivatives in $L^{1}$ then for all $\epsilon>0$ there is a constant 
$C$ depending on $\epsilon$ and the $L^{1}$
norms of the derivatives of $f_{1},f_{2}$ such that
$$
d_{TV}(X_{1},X_{2})\leq C W_{1}^{1-\epsilon}(X_{1},X_{2}).
$$
We now compare this result to ours: Corollary \ref{fourier2} above assumes the existence of all moments which Theorem 2.1 
in \cite{che} does not. Their condition of all
the derivatives being in $L^{1}$ also seems to be slightly weaker than \eqref{pipi}. However, $d_{FM}\leq W_{1}$, so our conclusion
in Corollary \ref{fourier2} is somewhat stronger than that of Theorem 2.1 in \cite{che}.
More importantly, our approach works for an arbitrary dimension $d$. 
\end{remark}

\begin{remark}\label{DK} {\rm For an integer $k\geq 1$,
denote by $\mathcal{M}_{k}$ the set of $k$ times continuously differentiable functions $f:\mathbb{R}^{d}\to\mathbb{R}$
such that $\sum_{l=0}^{k}||f^{(l)}||_{\infty}\leq 1$. 
Here $f^{(l)}$ refers to the $l$th derivative of $f$, in particular, $f^{(0)}=f$ and $f^{(1)}=\nabla f$. 
Following \cite{bally,bc1}, define the metric 
$$
d_{k}(\mu_{1},\mu_{2}):=\sup_{f\in\mathcal{M}_{k}}\left|\int_{\mathbb{R}^{d}}f(x)\mu_{1}(dx)-\int_{\mathbb{R}^{d}}f(x)\mu_{2}(dx)\right|.
$$
Note that $d_{1}=d_{FM}$. Setting $A:=2d_{k}(X_{1},X_{2})$, one may 
replace $\int_{|u|\leq M} (1+|u|) A\, du$ by $\int_{|u|\leq M} (1+|u|)^{k} A\, du$ in \eqref{modi} below, in the
proof of Lemma \ref{withrate}. Continuing that argument mutatis mutandis,
it follows that, under the conditions of Lemma \ref{withrate}, the estimate
$$
d_{TV}(X_{1},X_{2})\leq G d_{k}^{g}(X_{1},X_{2})	
$$
holds with 
$$
g=\frac{\gamma\delta}{(d+k+\gamma)(d+\delta)}
$$
and with a suitable constant $G$. Also Corollary \ref{fourier2} holds with $d_{FM}$ replaced by
$d_{k}$, for any $k$.}
\end{remark}

\begin{remark}\label{cf}
{\rm We may define, as in \cite{bc2}, the metric
$$
d_{CF}(\mu_{1},\mu_{2}):=\sup_{u\in\mathbb{R}^{d}}\left|\int_{\mathbb{R}^{d}}e^{i\langle u,x\rangle}\mu_{1}(dx)-{}
\int_{\mathbb{R}^{d}}e^{i\langle u,x\rangle}\mu_{2}(dx)\right|.
$$

Set $A:=d_{CF}(X_{1},X_{2})$. Then one may write $\int_{|u|\leq M}  A\, du$ instead of
$\int_{|u|\leq M} (1+|u|) A\, du$ in \eqref{modi}. Repeating the arguments of Lemma \ref{withrate} mutatis mutandis,
one arrives at 
$$
d_{TV}(X_{1},X_{2})\leq G d_{CF}^{g}(X_{1},X_{2})	
$$
with 
$$
g=\frac{\gamma\delta}{(d+\gamma)(d+\delta)}
$$
and with a constant $G>0$. Also Corollary \ref{fourier2} holds with $d_{FM}$ replaced by $d_{CF}$.}
\end{remark}

We finally mention one more new observation: a ``dominated convergence''-type criterion which allows to infer
total variation convergence from convergence in law under the mild condition \eqref{domin} but 
without providing a convergence rate.

\begin{lemma}\label{simple}
Let $X$, $X_{n}$, $n\in\mathbb{N}$ be $\mathbb{R}^{d}$-valued
random variables with $X_{n}\to X$ in law as $n\to\infty$. Let $\phi_{n}$ denote the
characteristic function of $X_{n}$. 
Assume that
\begin{equation}\label{domin}
\sup_{n\in\mathbb{N}}|\phi_{n}(u)|\leq \psi(u),\ u\in\mathbb{R}^{d}\mbox{ and }
\int_{\mathbb{R}^{d}}\psi(u)\, du<\infty
\end{equation}
for some measurable $\psi:\mathbb{R}^{d}\to\mathbb{R}_{+}$. 
Then $d_{TV}(X_{n},X)\to 0$, $n\to\infty$.
\end{lemma}

\begin{example} {\rm Let $Z(n)$, $n\in\mathbb{N}$ be a sequence of independent real-valued random variables
satisfying the conditions of Kolmogorov's three-series theorem. If there exists $n_{0}$
such that $\phi_{Z(n_{0})}$ is integrable then $\sum_{i=0}^{\infty}Z({i})$ converges in
total variation, too. Indeed, denoting $S_{n}:=\sum_{i=0}^{n}Z(i)$, we have
$$
|\phi_{S_{n}}(u)|=\prod_{i=0}^{n}|\phi_{Z(i)}(u)|\leq |\phi_{Z(n_{0})}(u)|
$$
for each $n\geq n_{0}$ since the absolute value of a characteristic function is at most $1$.
We may now invoke Lemma \ref{simple} above.}
\end{example}

\section{Contractive stochastic systems with memory}\label{contra} 

Consider a stationary sequence of random variables $\xi_{n}$ and a function $f(x,z)$. Stochastic processes
generated by the recursion $X_{n+1}=f(X_{n},\xi_{n+1})$ are called \emph{stochastically recursive sequence}s.
These far-reaching generalizations of Markov chains as well as their continuous-time generalizations appear in several application areas, 
see e.g.\ \cite{hairer,red,paolo-fad,borovkov-foss,elton}. 
Chapter 3 of the monograph \cite{borovkov} presents various contractivity 
hypotheses under which such sequences converge to an equilibrium state 
weakly, see also the recent survey \cite{gyorfi_laci}. Our purpose here is to show convergence in total variation
in the specific model of \cite{varvenne}, for
the first time in the literature.

Let $\varepsilon_{i}$, $i\in\mathbb{Z}$ be a sequence of $\mathbb{R}^{d}$-valued, integrable, i.i.d.\ random variables.
We shall consider the causal linear process 
$$
\xi_{n}:=\sum_{j=0}^{\infty}b_{j}\varepsilon_{n-j},\ n\in\mathbb{Z}
$$
where the scalars $b_{j}$ are assumed to satisfy 
\begin{equation}\label{coeff}
\sum_{j=0}^{\infty}|b_{j}|<\infty	
\end{equation} 
and $b_{0}\neq 0$.{}
Under these conditions, the series converges in $L^{1}$ as well as almost surely. If we
even stipulate that $\varepsilon_{0}$ is zero-mean and $E[|\varepsilon_{0}|^{2}]<\infty$ then condition
\eqref{coeff} can be weakened to $\sum_{j=0}^{\infty}b_{j}^{2}<\infty$. When we refer to
linear processes in the sequel we shall mean any of the two cases just explained.

Let $\nu:\mathbb{R}^{d}\to \mathbb{R}^{d}$ be measurable. 
We investigate the following stochastic equation, driven by the noise sequence
$\xi_{n}$, $n\in\mathbb{Z}$:
\begin{equation}\label{eqi}
X_{n+1}=\nu(X_{n})+\xi_{n+1},\ n\in\mathbb{N}.
\end{equation}
Note that $X_{n}$ has a non-linear dynamics and fails the Markov property, in general. It is a stochastic
dynamical system driven by \emph{coloured noise}.

\begin{assumption}\label{contractivity}
Let $X_{0}$ be integrable and independent of $(\varepsilon_{i})_{i\in\mathbb{Z}}$.
Let
\begin{equation}\label{contract}
|\nu(x_{1})-\nu(x_{2})|\leq \kappa |x_{1}-x_{2}|,\ x_{1},x_{2}\in\mathbb{R}^{d}	
\end{equation} 
hold for some $\kappa<1$. 
\end{assumption}

We recall the well-known Wasserstein-$1$ metric, defined for Borel probabilities $\mu_{1},\mu_{2}$ on $\mathbb{R}^{d}$
as follows:
$$
W_{1}(\mu_{1},\mu_{2}):=\inf_{\pi\in \mathcal{C}(\mu_{1},\mu_{2})}\int_{\mathbb{R}^{d}\times\mathbb{R}^{d}}
|x-y|\pi(dx,dy);
$$
here $\mathcal{C}(\mu_{1},\mu_{2})$ denotes the set of probabilities on $\mathbb{R}^{d}\times\mathbb{R}^{d}$
that are couplings of $\mu_{1},\mu_{2}$. Recall also the dual characterization (e.g.\ Theorem 5.10 in \cite{villani})
\begin{equation}\label{kr}
W_{1}(\mu_{1},\mu_{2})=\sup_{f\in\mathcal{N}}\left|\int_{\mathbb{R}^{d}}f(x)\mu_{1}(dx)-\int_{\mathbb{R}^{d}}f(x)\mu_{2}(dx)\right|,
\end{equation}
where $\mathcal{N}$ is the family of Lipschitz continuous $f:\mathbb{R}^{d}\to\mathbb{R}$ with Lipschitz constant at most $1$.

There are innumerable variants of the following result in the literature, often under much weaker assumptions.
We present only a very simple version here which is still sufficient to illustrate the power of the methods we have developed.
Greater generality could be sought at the price of more complicated conditions and arguments.

\begin{proposition}\label{limd} Let Assumption \ref{contractivity} be in force, that is, let \eqref{contract} hold with $\kappa<1$.
There exists $C>0$ and a probability $\mu_{*}$ such that
$$
W_{1}(\mathcal{L}(X_{n}),\mu_{*})\leq C\kappa^{n},\ n\in\mathbb{N}.
$$
\end{proposition}

Using Lemma \ref{withrate} above, we shall considerably strengthen Proposition \ref{limd} and
claim geometric convergence in total variation, too.

\begin{theorem}\label{lime} Assume that the characteristic function $\phi$ of $\varepsilon_{0}$
satisfies
\begin{equation}\label{fiu}
|\phi(u)|\leq C_{\flat}(1+|u|)^{-d-\gamma}
\end{equation} for some $\gamma, C_{\flat}>0$. Then 
$$
d_{TV}(\mathcal{L}(X_{n}),\mu_{*})\leq C'\varrho^{n},\ n\in\mathbb{N}.
$$	
for suitable $C'>0$, $\varrho<1$.
\end{theorem}

\begin{remark}
{\rm One can check, Theorem 1.8 of \cite{fourier}, that if $\varepsilon_{0}$ has a $(d+1)$ times
differentiable density with all these derivatives decaying at faster than polynomial rates at infinity then \eqref{fiu} holds, this
is essentially the argument used in Lemma \ref{majd} below. We can thus see that \eqref{fiu} is a mild requirement.

As the smoothness of the random variables $\xi_{n}$ is determined by that of $\varepsilon_{0}$,
the interpretation of condition \eqref{fiu} is that the coloured noise $\xi_{n}$ driving the system
should be ``smooth enough''.}	
\end{remark}

\begin{remark}
{\rm Continuous-time versions of the stochastic dynamical system \eqref{eqi}
were studied in \cite{hairer} and in subsequent papers. Fractional Gaussian noise 
was considered and the existence of a limiting law was established in total variation (actually,
in a much stronger sense). Instead of the strong contraction property \eqref{contract},   
only a dissipativity condition was imposed on $\nu$. 

Based on techniques of \cite{hairer}, \cite{varvenne} established
analogous results in discrete time for \eqref{eqi} when the $\varepsilon_{i}$ are Gaussian.
For several particular choices of the coefficients $b_{j}$ (exponential or power decay), \cite{varvenne} shows convergence
to a limiting law in total variation at certain rates, depending on the sequence $b_{j}$. 
Note that only polynomial rates are established in that paper.

Although our Theorem \ref{lime} is less deep (due to the strong Assumption \ref{contractivity}), 
it applies nevertheless to non-Gaussian systems and gives exponentially fast convergence in total variation.
We are not aware of any other such results in the existing literature.}
\end{remark}

\section{A Berry--Esseen-type result}\label{berry}

Although normal approximation is not our chief concern in the present research, we
demonstrate in this section what can be achieved in that area by our methods.

Let $Y_{n}$, $n\in\mathbb{N}$ be a sequence of i.i.d.\ $\mathbb{R}^{d}$-valued zero-mean random variables
with unit covariance matrix. We consider the normalized sums
$S_{n}:=\sum_{j=1}^{n}Y_{j}/\sqrt{n}$ and study their convergence speed to standard Gaussian law.
The optimal rate is well-known to be of order $O(n^{-1/2})$ in several metrics related to weak convergence,
see e.g.\ \cite{yurinskii} and the references therein.
In Theorem \ref{be} below, we achieve an \emph{almost} optimal rate in total variation distance 
under sufficiently strong smoothness and moment conditions.

\begin{assumption}\label{forclt}
Denote by $\phi$ the characteristic function of $Y_{0}$. There is $\alpha>0$ such that
\begin{equation}\label{smoothie}
|\phi(u)|\leq C_{\sharp}(1+|u|)^{-\alpha},\ u\in\mathbb{R}^{d}
\end{equation}
with some constant $C_{\sharp}$. For all $k\in\mathbb{N}$, 
\begin{equation}\label{momenti}
E[|Y_{0}|^{k}]<\infty.	
\end{equation}
\end{assumption}

\begin{theorem}\label{be} Let Assumption \ref{forclt} be in force. Then for each $\epsilon>0$, 
\begin{equation}\label{television}
d_{TV}(S_{n},\Phi)\leq C_{\epsilon}n^{-\frac{1}{2}+\epsilon},\ n\geq 2
\end{equation}
holds for a suitable constant $C_{\epsilon}$, where $\Phi$ is a random variable with $d$-dimensional standard Gaussian law.
One also has
\begin{equation}\label{supy}
\sup_{x\in\mathbb{R}^{d}}|f_{S_{n}}(x)-f_{\Phi}(x)|\leq C_{\epsilon}'n^{-\frac{1}{2}+\epsilon}
\end{equation}
where $f_{S_{n}}$ denotes the density function of $S_{n}$ and $f_{\Phi}$ is the standard $d$-dimensional
Gaussian density.
\end{theorem}

\begin{remark}
{\rm Assuming that $Y_{0}$ has a bounded density and finite third moment, Theorem 1.1 of the very recent
\cite{bobkov-goetze}
proves \eqref{supy} with $\epsilon=0$. 
Assume now also that $E[|Y_{0}|^{k}]<\infty$ for all $k$. Then for $M=n^{1/(4k+2)}$, Markov's inequality and \eqref{supy} imply
\begin{eqnarray*}
& & \int_{\mathbb{R}^{d}}|f_{S_{n}}(x)-f_{\Phi}(x)|\, dx\leq \int_{|x|\leq M} C'n^{-1/2}\, dx 
+\int_{|x|>M} [f_{S_{n}}(x)+f_{\Phi}(x)]\, dx\\
&\leq& C'n^{-1/2}M+E[|Y_{1}|^{2k}+|\Phi|^{2k}]M^{-2k}\leq C_{k}n^{-1/2 +1/(4k+2)}
\end{eqnarray*} 
for some constants $C'$, $C_{k}$ which, for $k$ satisfying $1/(4k+2)<\epsilon$,
shows that \eqref{television} holds under these conditions.
Our Theorem \ref{be} complements this observation, assuming \eqref{smoothie} instead of a bounded
density.

We do not know if, in a general multidimensional setting, one can obtain total variation rates better than $O(n^{-1/2+\epsilon})$.
In the case $d=1$ there exist more precise results: \cite{bcg} proves the optimal rate $O(n^{-1/2})$ in total
variation when $d=1$, under
the (mild) condition that the law of $Y_{0}$ has finite relative entropy with respect to
standard Gaussian law. 

Normal approximation in total variation has been thoroughly studied in the
context of Malliavin calculus, see \cite{nourdin-peccati,bally,bc2} and the references therein.

Note that our arguments for proving Theorem \ref{be} are quite simple,
they combine well-known results about
weak convergence estimates (such as \cite{yurinskii}) with the general-purpose Corollary \ref{fourier2} above
which does not rely on the specificity of the Gaussian distribution. With this in mind, Theorem
\ref{be} illustrates quite convincingly the strength of the methods developed in Section \ref{tot} above.
}
\end{remark}



\begin{example}
{\rm Let $d=1$. Remember that zero-mean Laplace distribution with parameter $\lambda>0$ has density 
$\lambda e^{-\lambda |x|}/2$, $x\in\mathbb{R}$ and characteristic function $\lambda^{2}/(\lambda^{2}+t^{2})$, $t\in\mathbb{R}$.
Recalling that $\sum_{k=1}^{\infty}\frac{6}{k^{2}\pi^{2}}=1$, we consider a mixture of Laplace distributions with parameters $k$:
$$
f(x):=\sum_{k=1}^{\infty}\frac{6}{k^{2}\pi^{2}}\frac{k}{2}e^{-k|x|}.
$$
Note that the series converges for all $x\neq 0$ to a density function which is unbounded near $0$ and whose characteristic function equals
$$
\phi(t):=\sum_{k=1}^{\infty}\frac{6}{k^{2}\pi^{2}}\frac{k^{2}}{k^{2}+t^{2}}\leq \frac{C}{1+|t|},\ t\in\mathbb{R},
$$ 
with some $C$, by elementary estimates.
This means that $\phi$ satisfies condition \eqref{smoothie}. 
Let $Z_{n}$ be an i.i.d.\ sequence with density $f$. Clearly, $E[Z_{1}]=0$ and
\eqref{momenti} holds for the sequence $Z_{n}$. Denoting by $\sigma^{2}$ the variance of $Z_{1}$, define $Y_{n}:=Z_{n}/\sigma$.
Theorem \ref{be} applies to the sequence $Y_{n}$, $n\in\mathbb{N}$ while
results of \cite{bobkov-goetze} do not, since $f$ is unbounded.
}	
\end{example}


\section{Applications to Malliavin calculus}\label{malliavin}

We assume that the reader is familiar with basic notions of Malliavin calculus on the Wiener space, as presented
in \cite{nualart} or in Chapter 2 of \cite{bally}. 
We restrict ourselves to the time interval $[0,1]$ and consider an $m$-dimensional Brownian motion $B_{t}$, $t\in [0,1]$. Standard notation 
will be applied, conforming to the
cited works. The set of $k$ times Malliavin-differentiable real-valued functionals with $p$-integrable derivatives
is denoted by $\mathbb{D}^{k,p}$, the cases $p=\infty$ or $k=\infty$ being self-explanatory. 

Let $G\in\mathbb{D}^{k,p}$ and let $\alpha=(i_{1},\ldots,i_{m})\in\mathbb{N}^{m}$ be a multiindex with $|\alpha|\leq k$, where
$|\alpha|:=i_{1}+\ldots+i_{m}$. Then
$D^{\alpha}G$ denotes the respective Malliavin (partial) derivative. Note that this is a \emph{random field}
indexed by $(t_{1},\ldots,t_{|\alpha|})\in [0,1]^{|\alpha|}$. 
We will need $\mathbb{R}^{d}$-valued functionals $F$
in the sequel, which are elements of $(\mathbb{D}^{k,p})^{d}$. The same notation $D^{\alpha}$ will refer to
their Malliavin derivatives. We will use the following norms:
$$
\Vert F\Vert_{k,p}:=||F||_{L^{p}(\Omega)}+\sum_{1\leq |\alpha|\leq k} \left\Vert \Vert D^{\alpha}F\Vert_{L^{2}([0,1]^{|\alpha|})}\right\Vert_{L^{p}(\Omega)},
$$ 
where $L^{p}(\Omega)$ refers to the Banach space of $p$-integrable $\mathbb{R}^{d}$-valued 
random variables while for integers $j\geq 1$, $L^{2}([0,1]^{j})$ is the Hilbert space of $\mathbb{R}^{d}$-valued
square-integrable functionals with respect to the $j$-dimensional Lebesgue measure on $[0,1]^{j}$.

The following result is well-known in various formulations, it is a simple
consequence of e.g.\ Theorem 2.3.1 of \cite{bally}.

\begin{proposition}\label{malliavin-estimate} Let $I$ be an arbitrary
index set. Let $q\geq 1$ be an integer and
let $F_{i}\in (\mathbb{D}^{q+2,\infty})^{d}$, $i\in I$ be a family of functionals
with corresponding Malliavin matrices $\sigma_{i}$.
Let 
$$
C_{r}^{\sigma}:=\sup_{i\in I}E[|\mathrm{det}(\sigma_{i})|^{-r}]<\infty,
$$
as well as
$$
C_{r}^{F}:=\sup_{i\in I}\Vert F_{i}\Vert_{q+2,r}<\infty,
$$
hold for all $r\geq 1$. Then $F_{i}$ has a continuous density $f_{i}$
with respect to the $d$-dimensional Lebesgue measure,
and for all $l\geq 1$ we have
\begin{equation}\label{kelli}
\max_{|\alpha|\leq q}\sup_{i\in I}\sup_{x\in\mathbb{R}^{d}}|\partial_{\alpha} f_{i}(x)|(1+|x|)^{l}\leq \bar{C}_{q,l}
\end{equation}
for some finite $\bar{C}_{q,l}$.
\end{proposition}

The following result follows from Lemma 3.9 and Proposition 3.15 in \cite{bc2} if we replace $d_{FM}$
by $W_{1}$. Our version is thus slightly stronger since $d_{FM}\leq W_{1}$. See
also \cite{bc1} for earlier related results.

Our proof is quite different from that of \cite{bc2}: they apply advanced Malliavin calculus,
we use only the fairly standard Proposition \ref{malliavin-estimate}, combined with Corollary \ref{fourier2} above
which is a general probabilistic result without any reference to Malliavin calulus.

\begin{theorem}\label{smoothsequence1} 
Let $F,F_{n}\in (\mathbb{D}^{\infty,\infty})^{d}$, $n\in \mathbb{N}$ be a family of functionals
with corresponding Malliavin matrices $\sigma,\sigma_{n}$.
Assume that, for all $r>0$, $q\geq 1$,
$$
E[|\mathrm{det}(\sigma)|^{-r}]+\sup_{n\in \mathbb{N}}E[|\mathrm{det}(\sigma_{n})|^{-r}]<\infty,
$$
and also 
$$
\Vert F\Vert_{q,r}+\sup_{n\in\mathbb{N}}\Vert F_{n}\Vert_{q,r}<\infty.
$$	
Then for all $\epsilon>0$ there is a constant $C_{\epsilon}$ such that
$$
d_{TV}(F_{n},F)\leq C_{\epsilon}d_{FM}^{1-\epsilon}(F_{n},F).
$$
Furthermore, for the respective densities $f,f_{n}$,
$$
\sup_{x\in\mathbb{R}^{d}}|f_{n}(x)-f(x)|\leq C_{\epsilon}d_{FM}^{1-\epsilon}(F_{n},F).
$$
\end{theorem}

\begin{remark}\label{romi}
{\rm One may replace $d_{FM}$ by $d_{CF}$ or by $d_{k}$ for any $k\geq 2$ in the above result, see the
proof of Theorem \ref{smoothsequence1}.
}
\end{remark}

\section{Ramifications}\label{ramif}

By demanding even more stringent conditions on the laws of the random variables in consideration,
it is possible to improve the estimate of Corollary \ref{fourier2} above.

\begin{lemma}\label{exx}
Let $X_{1},X_{2}$ be such that, for some $r,C_{r}>0$, 
\begin{equation}\label{kati} 
E[e^{r|X_{i}|}]\leq C_{r},\quad  
\int_{\mathbb{R}^{d}}|\phi_{i}(u)|e^{r|u|}\, du\leq C_{r},\ i=1,2.
\end{equation}
Then there are constants $G,G'$, depending only on $d,r,C_{r}$, such that
\begin{equation}\label{otto}
d_{TV}(X_{1},X_{2})\leq G d_{FM}(X_{1},X_{2})|\ln d_{FM}(X_{1},X_{2})|^{2d+1}
\end{equation}
and 
\begin{equation}\label{toto}
\sup_{x\in\mathbb{R}^{d}}|f_{1}(x)-f_{2}(x)|\leq G' d_{FM}(X_{1},X_{2})|\ln d_{FM}(X_{1},X_{2})|^{d+1}.
\end{equation}
\end{lemma}

\begin{remark}{\rm Under the conditions of Lemma \ref{exx},
one can similarly obtain that, for each $k\geq 1$,
$$
d_{TV}(X_{1},X_{2})\leq C_{k} d_{k}(X_{1},X_{2})|\ln d_{k}(X_{1},X_{2})|^{2d+k}
$$
hold for suitable constants $C_{k}$. Note also that for a suitable $C_{0}$,
$$
d_{TV}(X_{1},X_{2})\leq C_{0} d_{CF}(X_{1},X_{2})|\ln d_{CF}(X_{1},X_{2})|^{2d}.
$$
}
\end{remark}

\begin{remark}
{\rm $E[e^{r|X_{i}|}]\leq C_{r}$, $i=1,2$ in \eqref{kati} above is simply exponential integrability
of the random variables concerned. It is less clear how to check the condition on the characteristic functions. 
Assume $X_{1},X_{2}$ have densities $f_{1},f_{2}$ that have holomorphic extensions $F_{1},F_{2}$
on a strip of the form $\{z=iy+x:\ x\in\mathbb{R},\ y\in [-B,B]\}\subset\mathbb{C}$ satisfying $|F_{1}(iy+x)|+|F_{2}(iy+x)|\leq g(x)$, 
$y\in [-B,B]$ with 
some square-integrable $g$. Then a version of the celebrated Paley-Wiener theorem (see Theorem IV of \cite{pw}) 
easily implies that, for some $\bar{r},C_{\bar{r}}>0$, 
$$
\int_{\mathbb{R}^{d}}|\phi_{i}(u)|e^{\bar{r}|u|}\, du\leq C_{\bar{r}},\ i=1,2.
$$
}	
\end{remark}

\section{Proofs}\label{pofs}

\begin{proof}[Proof of Lemma \ref{withrate}] Note that the area of the $d$-dimensional sphere is $a_{d}=2\pi^{d/2}/\Gamma(d/2)$.{}
We may thus estimate
\begin{eqnarray}\nonumber
& & \int_{|u|\geq h}(1+|u|)^{-d-\gamma}\, du\leq \int_{|u|\geq h}|u|^{-d-\gamma}\, du\\
&=& \int_{r\geq h}a_{d}r^{d-1}r^{-d-\gamma}\, dr= 
\frac{C_{\gamma}}{h^{\gamma}}	
\label{explin}
\end{eqnarray}
holds for all $h>0$ with constant $C_{\gamma}:=a_{d}/\gamma$.

Notice that, $\phi_{i}$, $i=1,2$ being integrable on $\mathbb{R}^{d}$, their inverse Fourier-transforms
$f_{i}$, $i=1,2$ are the densities of $X_{i}$, $i=1,2$ with respect to the $d$-dimensional Lebesgue measure.
For each $u\in\mathbb{R}^{d}$, the functions 
$$
\psi_{u}(x):=\frac{\sin(\langle u,x\rangle)}{2(1+|u|)},\quad \chi_{u}(x):=\frac{\cos(\langle u,x\rangle)}{2(1+|u|)}$$
satisfy
$$
||\psi_{u}||_{\infty}+||\nabla \psi_{u}||_{\infty}\leq 1,\ ||\chi_{u}||_{\infty}+||\nabla \chi_{u}||_{\infty}\leq 1,
$$
hence, by definition of the Fortet-Mourier metric,
\begin{eqnarray*}
|\phi_{1}(u)-\phi_{2}(u)| &\leq& 2(1+|u|)|iE[\psi_{u}(X_{1})] +E[\chi_{u}(X_{1})]-iE[\psi_{u}(X_{2})] -E[\chi_{u}(X_{2})]|
\\ &\leq&{}
4(1+|u|)d_{FM}(X_{1},X_{2}).
\end{eqnarray*}
Let us write $A:=4 d_{FM}(X_{1},X_{2})$ henceforth. 
First let us consider
the case $A\leq 1$. 
By what has been said, 
\begin{equation}\label{u}
|\phi_{1}(u)-\phi_{2}(u)|\leq [1+|u|]A.
\end{equation}
Fix $M\geq 1$ and estimate, using the inverse Fourier transform, \eqref{fi}, \eqref{explin} and \eqref{u}:
\begin{eqnarray} \nonumber
|f_{1}(x)-f_{2}(x)| &\leq&	\frac{1}{(2\pi)^{d}}\int_{\mathbb{R}^{d}}|\phi_{1}(u)-\phi_{2}(u)|\, du\\
\label{modi} &\leq& \int_{|u|\leq M} (1+|u|) A\, du+\int_{|u|> M}[|\phi_{1}(u)|+|\phi_{2}(u)|]\, du \\
\nonumber &\leq& (1+M)A (2M)^{d}+\int_{|u|> M}2c_{\phi}(1+|u|)^{-d-\gamma}\, du\\
\nonumber &\leq& A2^{d+1}M^{d+1}+2C_{\gamma}c_{\phi}M^{-\gamma}.
\end{eqnarray}
Now set $M:=A^{-1/(d+1+\gamma)}$ and conclude that, for a suitable constant $\tilde{G}$,
\begin{equation}\label{madar}
|f_{1}(x)-f_{2}(x)| \leq \tilde{G}A^{\gamma/(d+1+\gamma)},\quad x\in\mathbb{R}^{d}.
\end{equation}
Continuing estimation using \eqref{ef} and the Markov inequality,
\begin{eqnarray*}
\int_{\mathbb{R}^{d}}|f_{1}(x)-f_{2}(x)|\, dx &\leq& \int_{|x|<M} \tilde{G} A^{\gamma/(d+1+\gamma)}\,dx +
\int_{|u|\geq M}[f_{1}(x)+f_{2}(x)]\, dx\\
&\leq& (2M)^{d}\tilde{G}A^{\gamma/(d+1+\gamma)}+P(|X_{1}|\geq M)+P(|X_{2}|\geq M)\\
&\leq& 
2^{d}M^{d}\tilde{G}A^{\gamma/(d+1+\gamma)}+E[|X_{1}|^{\delta}+|X_{2}|^{\delta}]M^{-\delta}\\
&\leq& 2^{d}M^{d}\tilde{G}A^{\gamma/(d+1+\gamma)}+ 2c_{f}M^{-\delta}.
\end{eqnarray*}
Setting $M:=A^{-\frac{\gamma}{(d+1+\gamma)(d+\delta)}}$, we conclude that
$$
d_{TV}(X_{1},X_{2})\leq G'A^{\frac{\gamma\delta}{(d+1+\gamma)(d+\delta)}} 
$$
for some constant $G'$.

In the case $A>1$ recall that $d_{TV}\leq 2$ always hence $d_{TV}(X_{1},X_{2})\leq 2A^{g}$. 
Finally, we can set $G:=2^{2g+1}+4^{g}G'$ and \eqref{alenyeg} follows. \eqref{sup} comes from
\eqref{madar} when $A\leq 1$. When $A>1$, \eqref{sup} follows from the fact that $f_{1},f_{2}$ are bounded by a constant 
depending on $d,\gamma$ and $c_{\phi}$. 	
\end{proof}

\begin{proof}[Proof of Corollary \ref{fourier2}] Choose $k,l$ so large that $\frac{kl}{(d+1+l)(d+k)}>1-\epsilon$. Now invoke Lemma 
\ref{withrate} with the choice $\delta=k$, $\gamma=l$. 
The resulting constant $C_{\epsilon}$ will depend only on $K_{k},C_{l+d}$ and $d$.
\end{proof}

\begin{proof}[Proof of Lemma \ref{simple}]
Each $X_{n}$ has a continuous density $f_{n}$, the inverse Fourier transform of $\phi_{n}$.
Let $\phi$ be the characteristic function of $X$. Note that 
\begin{equation}\label{mathias}
|\phi(u)|\leq \psi(u),\ u\in\mathbb{R}^{d}	
\end{equation}
also holds since, by convergence in law, $\phi$ is the pointwise limit of the sequence $\phi_{n}$.
Hence $X$ also has a (continuous) density $f$, the inverse Fourier transform of $\phi$.
{}
Using the inverse Fourier transform we may estimate
$$
|f_{n}(x)-f(x)|\leq \frac{1}{(2\pi)^{d}}\int_{\mathbb{R}^{d}}|\phi_{n}(u)-\phi(u)|\, du\to 0
$$
as $n\to\infty$,
by the pointwise convergence of $\phi_{n}$ to $\phi$ and by Lebesgue's theorem. 
We may now conclude using Scheff\'e's theorem.
\end{proof}

\begin{proof}[Proof of Proposition \ref{limd}]
For $n\geq 1$, define $\bar{X}_{-n}(n):=X_{0}$ and $\bar{X}_{k+1}(n):=\bar{X}_{k}(n)+\nu(\bar{X}_{k}(n))+\xi_{k+1}$
for $k=-n,-n+1,\ldots,-1$. Now set $\tilde{X}_{n}:=\bar{X}_{0}(n)$. Note that the law of $\tilde{X}_{n}$ equals
that of $X_{n}$. 
Now, noticing $|\nu(x)|\leq \kappa |x|+|\nu(0)|$, one can estimate
$$
E[|X_{n+1}|]\leq \kappa E[|X_{n}|]+|\nu(0)|+E[|\xi_{n+1}|]=\kappa E[|X_{n}|]+|\nu(0)|+E[|\xi_{0}|],
$$
which leads to
$$
S:=\sup_{n\in\mathbb{N}}E[|X_{n}|]\leq E[|X_{0}|]+\frac{1}{1-\kappa}[|\nu(0)|+E[|\xi_{0}|]]<\infty.
$$
Note that for $n<m$ and $k=-n,\ldots,-1$,
$$
E[|\bar{X}_{k+1}(n)-\bar{X}_{k+1}(m)|]\leq E[|\nu(\bar{X}_{k}(n))-\nu(\bar{X}_{k}(m))|]\leq 
\kappa E[|\bar{X}_{k}(n)-\bar{X}_{k}(m)|]
$$
hence
\begin{eqnarray} \nonumber
W_{1}(X_{n},X_{m}) &\leq& E[|\tilde{X}_{n}-\tilde{X}_{m}|]\leq \kappa E[|\bar{X}_{-1}(n)-\bar{X}_{-1}(m)|]
\leq \cdots\\ 
\label{legfont} &\leq&{}
\kappa^{n}E[|\bar{X}_{-n}(n)-\bar{X}_{-n}(m)|]\leq 2\kappa^{n}S,
\end{eqnarray}
where we used that $X_{0}=\bar{X}_{-n}(n)$ and the law of $\bar{X}_{-n}(m)$ 
equals that of $X_{m-n}$.
Thus the sequence of the laws of ${X}_{n}$ is Cauchy in the complete metric $W_{1}$
and converges to some $\mu_{*}$. Since $W_{1}(\mathcal{L}({X}_{n}),\mu_{*})=\lim_{m\to\infty}W_{1}(X_{n},X_{m})$
we get the statement from \eqref{legfont}.
\end{proof}

\begin{proof}[Proof of Theorem \ref{lime}]
We check that the conditions of Lemma \ref{withrate} hold for $X_{n}$, $n\in\mathbb{N}$
with $\delta=1$ and $\gamma$ as in the statement of this theorem. The sequence
$X_{n}$ is bounded in $L^{1}$ by the argument of Proposition \ref{limd} above. 

From \eqref{kr}, $d_{FM}\leq W_{1}$. It remains to
estimate the tails of $\phi_{X_{n}}$. Notice that, by independence,
\begin{eqnarray}\nonumber
& & |\phi_{X_{n}}(u)|=|E[e^{i\langle u,X_{n}\rangle}]|\\
&=& \label{firi} |E[e^{i\langle u,X_{n-1}+\nu(X_{n-1})+\sum_{j=1}^{\infty}b_{j}\varepsilon_{n-j}\rangle}]
E[e^{i\langle u,b_{0}\varepsilon_{n}\rangle}]|\\
&\leq& \nonumber |\phi(b_{0}u)|\leq C_{\flat}(1+|u b_{0}|)^{-d-\gamma}\leq C_{\flat}\left(\frac{1}{|b_{0}|}+1
\right)^{d+\gamma}(1+|u|)^{-d-\gamma}	
\end{eqnarray}
since in \eqref{firi} the first expectation's absolute value $\leq 1$. The statement follows.  
\end{proof}

\begin{proof}[Proof of Theorem \ref{be}]
According to the unnumbered Theorem in Section 3 of \cite{yurinskii}, $d_{PR}(S_{n},\Phi)=O(n^{-1/2})$.
Here $d_{PR}$ refers to the Prokhorov metric
$$
d_{PR}(\mu,\nu):=\inf\{\epsilon>0: \mu(A)\leq \nu(A^{\epsilon})+\epsilon\mbox{ and }
\nu(A)\leq \mu(A^{\epsilon})+\epsilon\mbox{ for all }A\in\mathcal{B}(\mathbb{R}^{d})\}
$$
where $A^{\epsilon}:=\{x\in\mathbb{R}^{d}:\exists a\in A \mbox{ with }|a-x|<\epsilon\}$.
Now Corollary 2 of \cite{dudley} implies $d_{FM}(S_{n},\Phi)=O(n^{-1/2})$. 

It remains to show that the conditions
of Corollary \ref{fourier2} hold for the sequence $S_{n}$ (they are obvious for the standard Gaussian law). 
This is done in Lemmas \ref{boundy} and \ref{bundi} below.	
\end{proof}

\begin{lemma}\label{boundy}
If \eqref{momenti} holds then, for each integer $k\geq 1$,
$$
\mathcal{S}_{k}:=\sup_{n\geq 1}E[|S_{n}|^{2k}]<\infty.
$$
\end{lemma}
\begin{proof} It suffices to show this coordinatewise so
we assume $d=1$.
The Mar\-czin\-ki\-e\-wicz--Zygmund inequality implies that, for some constant $C_{k}$,
$$
E[|S_{n}|^{2k}]\leq C_{k}E \left[\left(\sum_{j=1}^{n} \frac{Y^{2}_{j}}{n}\right)^{k}\right]=
\frac{C_{k}}{n^{k}}E \left[\left(\sum_{j=1}^{n} Y^{2}_{j}\right)^{k}\right].
$$
The latter expectation is a sum of $n^{k}$ terms of the form $E[Y_{j_{1}}^{2}\cdots Y_{j_{k}}^{2}]$ with some
indices $j_{1},\ldots,j_{k}$. Since $E[|Y_{1}|^{2k}]<\infty$, clearly each such term is smaller than some
fixed constant $D_{k}$. Hence $E[|S_{n}|^{2k}]\leq {C_{k}}{n^{-k}}D_{k}n^{k}\leq C_{k}D_{k}$, for each $n$.
\end{proof}

\begin{lemma}\label{bundi} If \eqref{smoothie} holds then
for each integer $l\geq 1$, there are $N(l),H_{l}>0$ such that
\begin{equation}\label{valami}
|\phi_{S_{n}}(u)|\leq H_{l}(1+|u|)^{-l},\ u\in\mathbb{R}^{d},\ n\geq N(l).
\end{equation}
\end{lemma}
\begin{proof}
Notice that $\phi_{S_{n}}(u)=\phi^{n}(u/\sqrt{n})$. We will present three separate arguments
depending on how large $u$ is. Fix an integer $l\geq 1$.

\smallskip

\noindent\textsl{Case 1: $|u|$ small.} As $Y_{0}$ has zero mean and unit covariance matrix and $\phi$ is smooth
by \eqref{momenti}, we have $\phi(0)=1$, $\phi'(0)=0$, $\phi''(0)=-\mathrm{Id}_{\mathbb{R}^{d}}$ 
and there is $0<c_{1}\leq 1$ such that 
$\phi(x)=1-|x|^{2}/2+O(|x|^{3})$ for $|x|\leq c_{1}$. Choosing an even smaller $c_{1}$ we may thus
guarantee that $|\phi(x)|\leq 1-|x|^{2}/3$ for all $|x|\leq c_{1}$. Recalling the Taylor expansion
of the logarithm, we may estimate 
$$
\left(1-\frac{|x|^{2}}{3}\right)\leq e^{-|x|^{2}/4}
$$
for all $|x|\leq c_{2}$, with $c_{2}\leq c_{1}$ small enough, hence
$$
|\phi_{S_{n}}(u)|\leq \left(1-\frac{|u|^{2}}{3n}\right)^{n}\leq e^{-|u|^{2}/4}
$$
for all $n$ and for all $|u|\leq c_{2}\sqrt{n}$.

Now notice that $(1+x)^{l}\leq l!e^{x+1}\leq l! e^{\frac{x^{2}}{4} +3}$ for all $x\geq 0$,
hence also 
$$
|\phi_{S_{n}}(u)|\leq \frac{e^{3}l!}{(1+|u|)^{l}},
$$
showing \eqref{valami} for $|u|\leq c_{2}\sqrt{n}$.

\smallskip{}

\noindent\textsl{Case 2: $|u|$ large.} We may and will assume $C_{\sharp}\geq 1$. By \eqref{smoothie}, one has 
\begin{equation}\label{tavol}
|\phi(x)|\leq C_{\sharp}|x|^{-\alpha},\ x\in\mathbb{R}^{d}.	
\end{equation} 
Let $n\geq \hat{N}(l):=2l/\alpha$. 
Set $J:=C_{\sharp}^{2/\alpha}$. Notice that for $|u|\geq Jn$, one has
$$
|u|\geq n^{\frac{\alpha n}{2(\alpha n -l)}}C_{\sharp}^{\frac{n}{\alpha n-l}}
$$
but then, since \eqref{tavol} holds,
$$
|\phi_{S_{n}}(u)|\leq \frac{(\sqrt{n})^{\alpha n}C_{\sharp}^{n}}{|u|^{\alpha n} }\leq \frac{1}{|u|^{l}}.
$$
This implies 
$$
|\phi_{S_{n}}(u)|\leq \frac{2^{l}}{(1+|u|)^{l}}
$$
since $|u|\geq Jn\geq 1$.

\smallskip{}

\noindent\textsl{Case 3: $|u|$ moderate.}
Apply Lemma \ref{sublemma} below
with the choice $c:=c_{2}$. We obtain $\rho$ such that $|\phi_{S_{n}}(u)|\leq \rho^{n}$
for all $|u|\geq c_{2}\sqrt{n}$. 
It is clear that $\rho^{n}\leq (1+Jn)^{-l}$ holds for $n\geq \tilde{N}$ with some 
$\tilde{N}=\tilde{N}(l)$.
This shows \eqref{valami} for all $c_{2}\sqrt{n}\leq |u|\leq Jn$.

\smallskip

We have thus verified \eqref{valami} for all $l\geq 1$, $n\geq N(l):=\max\{\tilde{N}(l),\hat{N}(l)\}$ 
and $u\in\mathbb{R}^{d}$ and may conclude. 
\end{proof} 

\begin{lemma}\label{sublemma} Let \eqref{smoothie} hold. Then for each $c>0$,
there is $\rho=\rho(c)<1$ satisfying $\sup_{|u|\geq c}|\phi(u)|\leq \rho$.	
\end{lemma}
\begin{proof}
$\phi(u)\to 0$ as $u\to\infty$. Hence $\sup_{|u|\geq K}|\phi(u)|\leq 1/2$
for some $K>0$. $|\phi|$ being continuous, it achieves its maximum on the compact set 
$\{u:c\leq |u|\leq K\}$ at some point $u^{*}\neq 0$. We show that $|\phi(u^{*})|<1$.
By contradiction, in the opposite case the one-dimensional random variable $U:=\langle u^{*},Y_{0}\rangle${}
would satisfy $|\phi_{U}(1)|=|\phi(u^{*})|=1$. By Theorem 2.1.4 of
\cite{lukacs}, $U$ should have lattice distribution. As such, $|\phi_{U}|$ would be periodic but
this contradicts $\lim_{t\to\infty}\phi_{U}(t)=\lim_{t\to\infty}\phi(tu^{*})=0$. 
\end{proof}

\begin{proof}[Proof of Proposition \ref{malliavin-estimate}] First let us fix $i$ and apply Theorem 2.3.1 of \cite{bally} to $F_{i}$ with the
choice $p=2d$, $k=2(d-1)$ and $a=1/4d$ ($p,a,k$ here refer to the notation in \cite{bally}). Equation (2.88) of 
\cite{bally} implies
that for a suitable constant $C_{\natural}{}$ and for some $p'\geq 1$,
\begin{equation}\label{elsofontos}
|\partial_{\alpha} f_{i}(x)|\leq C_{\natural}\left[\left(1+\Vert\mathrm{det}(\sigma_{i}^{-1})\Vert_{p'}\right)^{q+2}
\left(1+\Vert F_{i}\Vert_{q+2,p'}\right)^{2d(q+2)}\right]^{q+1+k}
\end{equation}
holds for all $x$. Equation (2.89) of \cite{bally} implies that also
\begin{equation}\label{masikfontos}
|\partial_{\alpha} f_{i}(x)|\leq C_{\natural}\left[\left(1+\Vert \mathrm{det}(\sigma_{i}^{-1})\Vert_{p'}\right)^{q+2}
\left(1+\Vert F_{i}\Vert_{q+2,p'}\right)^{2d(q+2)}\right]^{q+1+k}\frac{\Vert F_{i}\Vert_{4dl}^{l}}{(|x|-2)^{l}}
\end{equation}
for all $|x|>2$. 
Taking supremum in $\alpha,x,i$ in \eqref{elsofontos} and \eqref{masikfontos} combined
with the hypotheses $C_{p'}^{\sigma},C_{p'}^{F}<\infty$ implies our statement since 
$(|x|-2)^{-l}\leq 4(1+|x|)^{-l}$, $|x|\geq 3$ and $(1+|x|)^{-l}\leq 4^{-l}$, $|x|<3$. 
\end{proof}

\begin{proof}[Proof of Theorem \ref{smoothsequence1}]
Proposition \ref{malliavin-estimate} and Lemma \ref{majd} below imply that the conditions
of Corollary \ref{fourier2} hold. Now the statements follow from that Corollary. Remark 
\ref{romi} follows from
Remarks \ref{DK} and \ref{cf} above.	
\end{proof}

\begin{lemma}\label{majd} Let the assumptions of Proposition \ref{malliavin-estimate} be in vigour.
Let $\phi_{i}$ denote the characteristic function of $F_{i}$ for $i\in I$. Then for each $l\geq 1$ there are constants $H_{l}$
such that for all $i\in I$,
\begin{equation}\label{dedi}
|\phi_{i}(u)|\leq H_{l}(1+|u|)^{-l},\ u\in\mathbb{R}^{d}.
\end{equation}
\end{lemma}
\begin{proof} For a function $g\in L^{1}$ we define its Fourier transform as $$
\hat{g}(t):=\int_{\mathbb{R}^{d}}g(x)e^{i\langle t,x\rangle}\, dx,\ t\in\mathbb{R}^{d}.
$$
Note that $\hat{f}_{i}=\phi_{i}$.
We will show that, for each multiindex $\alpha$, 
$$
(-i)^{|\alpha|}\phi_{i}(u)u_{1}^{\alpha_{1}}\ldots u_{d}^{\alpha_{d}}
$$ 
is the Fourier-transform of $\partial_{\alpha}f_{i}$. Since this is
is bounded by the $L_{1}$-norm of $\partial_{\alpha}f_{i}$, and the latter form a 
bounded family in $L^{1}$ by Proposition \ref{malliavin-estimate},
we may easily deduce \eqref{dedi}.
  
We will only deal with the first order derivative $\partial_{x_{1}}f_{i}$ since the general argument is analogous. Dropping
the index $i$ we simply write $f$.

Fix $x\in\mathbb{R}^{d}$ for the moment. Consider the expression $$\ell(x,h):=[f(x+he_{1})-f(x)]/h,$$ where $e_{1}=(1,0,\ldots,0)$, $h>0$. By the 
mean-value theorem, there are numbers $\xi_{x,h}\in [0,h]$ such that $\ell(x,h)=\partial_{x_{1}}f(x+\xi_{x,h}e_{1})$. We wish to apply
Theorem 1.8 of \cite{fourier} hence we need to verify
\begin{equation}\label{ros}
\int_{\mathbb{R}^{d}}|\ell(x,h)-\partial_{x_{1}}f(x)|\, dx\to 0\mbox{ as }h\to 0.
\end{equation}
From \eqref{kelli}, the integrand is dominated by
\begin{equation}\label{jo}
C_{1,d+1}^{F}[(1+|x+\xi_{x,h}e_{1}|)^{-d-1}+(1+|x|)^{-d-1}],
\end{equation}
where the second term is clearly integrable. Let $h<1/2$ hold, then
$$
1+|x+\xi_{x,h}e_{1}|\geq \frac{1+|x|}{2}
$$
for all $x\in\mathbb{R}^{d}$, 
which guarantees that the first term in \eqref{jo} is also integrable, hence dominated convergence works for \eqref{ros}.
\end{proof}
 
\begin{proof}[Proof of Lemma \ref{exx}]
Denote $A:=d_{FM}(X_{1},X_{2})$. First consider the case $A\leq e^{-r}$.
Using the inverse Fourier transform and Cauchy's inequality, we may estimate, as in Lemma \ref{withrate},
for an arbitrary $M\geq 1$,
\begin{eqnarray*}
& & |f_{1}(x)-f_{2}(x)|\leq \int_{\mathbb{R}^{d}}|\phi_{1}(u)-\phi_{2}(u)|\, du\\
&\leq& \int_{|u|\leq M}	4(1+|u|)A\, du+\int_{|u|>M}[|\phi_{1}(u)|+|\phi_{2}(u)|]\, du\\
&\leq& 4(M+1)A (2M)^{d}+ \int_{|u|>M}[|\phi_{1}(u)|+|\phi_{2}(u)|]e^{r|u|/2}e^{-r|u|/2}\, du\\
&\leq& 8M 2^{d}M^{d}A+\left({}
\int_{\mathbb{R}^{d}}2[|\phi_{1}(u)|^{2}+|\phi_{2}(u)|^{2}]e^{r|u|}\, du
\right)^{1/2}\left({}
\int_{|u|>M}e^{-r|u|}\, du
\right)^{1/2}\\
&\leq& C'\left(M^{d+1}A+\left[\int_{M}^{\infty}e^{-ry}y^{d-1}\, dy\right]^{1/2}\right)\\
&\leq& C''\left(M^{d+1}A+M^{(d-1)/2}e^{-rM/2}\right)
\end{eqnarray*}	
for suitable constants $C',C''$.
Choose $M:=2|\ln A|/r$ to obtain
\begin{equation}\label{aura}
|f_{1}(x)-f_{2}(x)|\leq C'''A|\ln A|^{d+1}
\end{equation}
with a suitable $C'''$. 
For any $M\geq 1$, we may estimate, using Markov's inequality,
\begin{eqnarray*}
& & \int_{\mathbb{R}^{d}}|f_{1}(x)-f_{2}(x)|\, dx\\
&\leq& \int_{|x|\leq M}C'''A |\ln A|^{d+1}\, dx +\int_{|x|>M} [f_{1}(x)+f_{2}(x)]\, dx\\
&\leq& C'''2^{d}M^{d}A |\ln A|^{d+1}+\left[E[e^{r|X_{1}|}+e^{r|X_{2}|}]\right]e^{-rM}\\
&\leq& C'''2^{d}M^{d}A |\ln A|^{d+1}+2C_{r}e^{-rM}.	
\end{eqnarray*} 
Choosing $M:=|\ln A|/r$, we arrive at
$$
d_{TV}(X_{1},X_{2})\leq C''''A|\ln A|^{2d+1}
$$
with a suitable constant $C''''$. Finally, if $A\geq e^{-r}$ then
$$
d_{TV}(X_{1},X_{2})\leq 2\leq 2Ae^{r},
$$
so we may set $G:=C''''+2e^{r}$ and \eqref{otto} follows. \eqref{toto} comes from \eqref{aura} in the case $A< e^{-r}$.
In the alternative case, note that, $\phi_{1},\phi_{2}$ being integrable, 
their inverse Fourier-transforms $f_{1},f_{2}$ are bounded by a constant depending on $d$ and $C_{r}$ only.
\end{proof}




\begin{thebibliography}{10}





\bibitem{bc1}
V. Bally and L. Caramellino.
\newblock On the distances between probability density functions. 
\newblock \emph{Electron.
J. Probab.}, 19, no. 110, 1--33, 2014.

\bibitem{bally} V. Bally and L. Caramellino.
\newblock Integration by parts formulas, Malliavin calculus and
regularity of probability laws.
\newblock \emph{In: V. Bally, L. Caramellino and R. Cont: Stochastic integration
by parts and functional It\^{o} calculus}, 1--114, Birkh\"auser, 2016.

\bibitem{bc2} V. Bally, L. Caramellino and G. Poly.
\newblock Regularization lemmas and convergence in total variation. 
\newblock \emph{Electron.
J. Probab.}, 25, no. 74, 1--20, 2020.

\bibitem{6} M. Barkhagen, N. H. Chau, \'E. Moulines, M. R\'asonyi, S. Sabanis and Y. Zhang.
\newblock On stochastic gradient Langevin dynamics with dependent data streams in the logconcave case.
\newblock \emph{Bernoulli}, 27:1--33, 2021.

\bibitem{bobkov}
S. G. Bobkov.
\newblock Proximity of probability distributions
in terms of Fourier--Stieltjes transforms.
\newblock \emph{Russian Math. Surveys}, 71:1021--1079, 2016.

\bibitem{bcg}
S. G. Bobkov, G. P. Chistyakov and F. G\"otze. 
\newblock Berry-Esseen bounds in the entropic central limit theorem. 
\newblock \emph{Probab. Theory Relat. Fields}, 159:435--478, 2014.

\bibitem{bobkov-goetze}
S. G. Bobkov and F. G\"otze.
\newblock Berry-Esseen bounds in local limit theorems.
\newblock \emph{Preprint, arXiv:2407.20744}, 2024.

\bibitem{borovkov-foss}
A. A. Borovkov, S. G. Foss. \newblock Stochastically recursive sequences and their
generalizations. \newblock \emph{Limit theorems for random processes and their applications} (in Russian), 
20:32--103, 1993.
 

\bibitem{borovkov} A. A. Borovkov. 
\newblock \emph{Ergodicity and Stability of Stochastic Processes.}
\newblock Wiley, Chichester, 1998. 

\bibitem{che} M. Chae and S.G. Walker. 
\newblock Wasserstein upper bounds of the total variation for smooth densities.
\newblock \emph{Statistics and Probability Letters}, 163, 108771, 2020.


\bibitem{dudley} R. M. Dudley. \newblock Distances of probability measures and random variables.
\newblock \emph{Ann. Math. Stat.}, 39:1563--1572, 1968.

\bibitem{elton}
J. H. Elton. \newblock A multiplicative ergodic theorem for Lipschitz maps.
\newblock {\em Stochastic Processes and its Applications}, 34:39--47, 1990.

\bibitem{paolo-fad} P. Guasoni.
\newblock Asymmetric Information in Fads Models.
\newblock \emph{Finance and Stochastics}, 10:159--177, 2006.

\bibitem{gyorfi_laci} L. Gy\"orfi, A. Lovas and M. R\'asonyi.{}
\newblock On the strong stability of ergodic iterations.
\newblock \emph{Journal of Applied Probability}, vol. 62, no. 1, 2025.

\bibitem{hairer} M. Hairer. 
\newblock Ergodicity of stochastic differential equations driven by fractional Brownian motion.
\newblock \emph{Annals of
Probability}, 32:703--758, 2005.

\bibitem{lukacs} E. Lukacs.
\newblock \emph{Characteristic functions. 2nd edition.}
\newblock Griffin, London, 1970. 

\bibitem{red} A. Morr, D. Kreher and N. Boers. 
\newblock Red noise in continuous-time stochastic modelling.
\newblock \emph{Royal Society Open Science}, 12(8), 250573, 2025.

\bibitem{nourdin-peccati} I. Nourdin and G. Peccati.
\newblock \emph{Normal approximations with Malliavin calculus: from Stein's method to
universality.}
\newblock Cambridge Tracts in Mathematics, vol. 192,
\newblock Cambridge University Press, 2012.

\bibitem{nualart} D. Nualart.
\newblock \emph{The Malliavin calculus and related topics.} 
\newblock 2nd edition, Springer, 2006.

\bibitem{pw} R. C. Paley and N. Wiener.
\newblock \emph{Fourier transforms in the complex domain.}
\newblock American Mathematical Society, New York, 1934.

\bibitem{fourier} E. M. Stein and G. Weiss.
\newblock \emph{Introduction to Fourier analysis on Euclidean spaces.}
\newblock Princeton University Press, 1971.

\bibitem{varvenne} M. Varvenne.
\newblock Rate of convergence to equilibrium for discrete-time
stochastic dynamics with memory.
\newblock \emph{Bernoulli}, 25:3234--3275, 2019.

\bibitem{villani}
C. Villani.
\newblock \emph{Optimal transport: old and new.}
\newblock Springer, Berlin, 2009.


\bibitem{yurinskii}
V. V. Yurinskii.
\newblock A smoothing inequality for estimates of the Levy--Prokhorov distance. 
\newblock \emph{Theory Probab. Appl.}, 20:1--10, 1975.






\end{thebibliography}
\end{document}